\documentclass[a4paper,12pt]{article}

\usepackage[utf8]{inputenc}
\usepackage[T1]{fontenc}
\usepackage[a4paper, margin=2.5cm]{geometry}
\usepackage{amsmath, amsthm, amssymb, enumerate}
\usepackage{lmodern,microtype,mathtools}
\usepackage{todonotes}
\usepackage{thmtools}
\usepackage{algpseudocode}
\usepackage{hyperref}
\usepackage{cleveref}
\usepackage{tikz}
\usepackage{thmtools}
\usepackage{thm-restate}

\declaretheorem[name=Theorem, numberwithin=section]{theorem}

\declaretheorem[name=Proposition, sibling=theorem]{proposition}
\declaretheorem[name=Observation, sibling=theorem]{observation}

\declaretheorem[name=Question, style = remark, sibling=theorem]{question}
\newenvironment{usethmcounterof}[1]{%
\renewcommand\thetheorem{\ref{#1}}\theorem}{\endtheorem\addtocounter{theorem}{-1}}
\newcounter{cas}

\newtheoremstyle{assert}
  {.5\baselineskip±.2\baselineskip}   
  {.5\baselineskip±.2\baselineskip}   
  {\itshape}  
  {0pt}       
  {\bfseries} 
  {.}         
  {5pt plus 1pt minus 1pt} 
  {(\thmnumber{#2})}          
\theoremstyle{assert}
\newtheorem{as}[cas]{}
\newcounter{scas}

\newtheoremstyle{sassert}
  {.5\baselineskip±.2\baselineskip}   
  {.5\baselineskip±.2\baselineskip}   
  {\itshape}  
  {0pt}       
  {\itshape} 
  {.}         
  {5pt plus 1pt minus 1pt} 
  {(\thmnumber{#2})}          
\theoremstyle{sassert}
\newtheorem{sas}[scas]{}

\renewcommand{\ge}{\geqslant}
\renewcommand{\geq}{\geqslant}
\renewcommand{\le}{\leqslant}
\renewcommand{\leq}{\leqslant}
\newcommand{\footremember}[2]{%
\footnote{#2}
\newcounter{#1}
\setcounter{#1}{\value{footnote}}%
}
\newcommand{\footrecall}[1]{%
\footnotemark[\value{#1}]%
}

\def\cqedsymbol{\ifmmode$\lrcorner$\else{\unskip\nobreak\hfil
\penalty50\hskip1em\null\nobreak\hfil$\lrcorner$
\parfillskip=0pt\finalhyphendemerits=0\endgraf}\fi} 
\newcommand{\cqed}{\renewcommand{\qed}{\cqedsymbol}}

\DeclareMathOperator{\ir}{ir}
\DeclareMathOperator{\mir}{mir}
\newcommand{\sst}[2]{\left\{#1\,:\,#2\right\}}
 
\title{Revisiting a theorem by Folkman on graph colouring\thanks{Authors MB and
PC have been supported by the ANR Project DISTANCIA (ANR-17-CE40-0015) operated by the French National Research Agency (ANR). PC has been supported by the ANR Project HOSIGRA (ANR-17-CE40-0022) operated by the French National Research Agency (ANR) and by INRIA GANG Project team. GJ is supported by an ARC grant from the Wallonia-Brussels Federation of Belgium.
OD, AL and VL have been supported by ANR Project \textsc{GraphEn} (ANR-15-CE40-0009)
operated by the French National Research Agency (ANR).%
}}

\author{%
  Marthe Bonamy\footnote{Service public français de la recherche, CNRS, LaBRI, Université de Bordeaux, France. \texttt{marthe.bonamy@u-bordeaux.fr}.}
\and
  Pierre Charbit\footnote{Service public français de la recherche, Université Paris Diderot -- IRIF, Paris, France. \texttt{charbit@irif.fr}.}
\and
  Oscar Defrain\footremember{auvergne}{Service public français de la recherche, Université Clermont Auvergne, France. \texttt{\{oscar.defrain, aurelie.lagoutte, vincent.limouzy, lucas.pastor\}@uca.fr}.}
\and
  Gwenaël Joret\footnote{Département d'Informatique, Université Libre de Bruxelles, Brussels, Belgium. \texttt{gjoret@ulb.ac.be}.}
\and
  Aurélie Lagoutte\footrecall{auvergne}
\and
  Vincent Limouzy\footrecall{auvergne}
\and
  Lucas Pastor\footrecall{auvergne}
\and
  Jean-Sébastien Sereni\footnote{Service public français de la recherche, Centre National de la Recherche Scientifique, ICube (CSTB), Strasbourg, France. \texttt{sereni@kam.mff.cuni.cz}.}
}

\begin{document}

\maketitle

\begin{abstract}%
We give a short proof of the following theorem due to Jon H.~Folkman (1969):
The~chromatic~number~of~any~graph is at most~$2$ plus the maximum over all
subgraphs of the difference between the number of vertices and twice the
independence number.
\end{abstract}

\section{Introduction}\label{sec:intro} 

Independent sets in a graph count among the most studied objects of graph
theory, both for their theoretical and practical appeal.  An \emph{independent
set} in a graph is a set of vertices no two of which are adjacent.  This notion
is at the heart of graph colouring, where one tries to find a partition of the
vertex set of a graph in the smallest possible number of independent sets, this
number being the \emph{chromatic number} of the graph. It follows that
every graph~$G$ with chromatic number~$k$ has an independent set of size at
least~$|V(G)|/k$.  We investigate what kind of converse statement could be
true.  As is usual, we write~$\chi(G)$ for the chromatic number of~$G$
and~$\alpha(G)$ for the \emph{independence number} of~$G$, that is, the size of
a largest independent set in~$G$. Since we study the independence number of a
graph in relation to the number of its vertices, it is useful to define the
\emph{independence ratio}~$\ir(G)$ of a graph~$G$ (with at least one vertex) to
be~$\frac{\alpha(G)}{|V(G)|}$, and the \emph{minimum
independence ratio}~$\mir(G)$ to be~$\min\sst{\ir(H)}{\text{$H$ induced
subgraph of~$G$ with at least one vertex}}$.  (The inverse of the minimum independence ratio is
sometimes called the Hall ratio.) Let us start with a straightforward
observation. Every graph with chromatic number greater than~$2$ contains an
induced odd cycle, and the independence number of an odd cycle is less than
half the number of its vertices.  The contrapositive statement reads as follows.

\begin{observation}\label{obs-bip}
If every induced subgraph~$H$ of a graph~$G$ has independence ratio
at least~$\frac12$, then~$G$ has chromatic number at most~$2$.
\end{observation}

One could try to generalise \Cref{obs-bip} in several ways. For
instance, what about replacing the constant~$2$ by some larger integer~$k$?
This would yield an incorrect statement, as for each integer~$k\ge3$, there
is a graph with chromatic number greater than~$k$ and minimum independence ratio at least~$\frac{1}{k}$.
Indeed, let~$M_k$ be the $k^\text{th}$ Mycielski graph (so~$M_2$ is a cycle of length~$5$).
Then~$\chi(M_k)=k+1$ and, as is well known, every proper subgraph of~$M_k$ has chromatic number
at most~$k$, from which one infers that~$\mir(M_k)\ge\frac{1}{k}$ whenever~$k\ge3$.

Let us point out here that the exact value of the minimum independence ratio of
Mycielski graphs seems to be unknown. In~2006, Cropper, Gyarfás and
Lehel~\cite{CGL06} proved that every triangle-free graph is an induced subgraph
of a Mycielski graph and, using results on Ramsey numbers~\cite{AKS80,Kim95}
and on the fractional chromatic number of Mycieslki graphs~\cite{LPU95}, they
inferred the existence of two positive real numbers~$c_1$ and~$c_2$ such that
for every integer~$m\ge3$,
\[ c_1\cdot\frac{\sqrt{\log m}}{m} \le \mir(M_s) \le c_2\cdot\frac{\log m}{m}.\]
Here~$s$ is one less than the Ramsey number~$R(3,m)$, which is the largest integer~$n$
such that there exists an $n$-vertex triangle-free graph with independence number less than~$m$.
As is well known,~$R(3,m)=\Theta(m^2/\log m)$ (the lower bound was established by Kim~\cite{Kim95}
in~1995, while the upper bound had been proved fifteen years before by Ajtai, Komlós and Szemerédi~\cite{AKS80};
simpler proofs for the lower bound along with improvements of the multiplicative constant
were found subsequently~\cite{AKS81,Gri83,She91}).

Another way is to look for approximate generalisations of \Cref{obs-bip}.
Let us say that a graph is \emph{half-stable} if its independence ratio is at least one half.
\begin{question}\label{que-half}
Let~$k$ be a non-negative integer and~$G$ a graph. Assume that for every induced
    subgraph~$H$ of~$G$, there exists a set~$Y\subseteq V(H)$ of at most~$k$
    vertices such that~$H-Y$ is half-stable. Is it true that there
    exists a subset~$X\subseteq V(G)$ of at most~$k$ vertices such that
    $G-X$ is bipartite?
\end{question}
\noindent
Again this turns out to be false: for $k=1$, consider the graph obtained
from the cycle~$v_1v_2v_3v_4v_5v_6$ of length~$6$ by adding the triangle
$v_1v_3v_5$ (see Figure~\ref{fig:CEque-half}). 

\begin{figure}[h]
  \centering
  \begin{tikzpicture}
    \tikzstyle{whitenode}=[draw,circle,fill=white,minimum size=9pt,inner sep=0pt]
    \tikzstyle{blacknode}=[draw,circle,fill=black,minimum size=7pt,inner sep=0pt]
    \tikzstyle{rednode}=[draw,circle,fill=red,minimum size=9pt,inner sep=0pt]
    \tikzstyle{greennode}=[draw,circle,fill=green,minimum size=9pt,inner sep=0pt]
    \tikzstyle{bluenode}=[draw,circle,fill=black!20!blue,minimum size=9pt,inner sep=0pt]
    \tikzstyle{texte}=[minimum size=0pt,inner sep=0pt]
 \draw (0,0) node[blacknode] (a) {}
-- ++(30+60:1cm) node[blacknode] (b) {}
-- ++(30+2*60:1cm) node[blacknode] (c) {}
-- ++(30+3*60:1cm) node[blacknode] (d) {}
-- ++(30+4*60:1cm) node[blacknode] (e) {}
-- ++(30+5*60:1cm) node[blacknode] (f) {};
\draw (f) -- (a) -- (c) -- (e) -- (a);
  \end{tikzpicture}
  \caption{A negative answer to Question~\ref{que-half} for $k=1$.}\label{fig:CEque-half}
\end{figure}
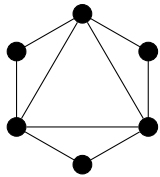

A more careful question was raised by Erd\H{o}s and Hajnal.
A graph satisfying the hypothesis of \Cref{que-half} has been
called \emph{$k$-near-bipartite} in the literature~\cite{Ree99}.
Erd\H{o}s and Hajnal conjectured (see~\cite{Gya97}) the existence of
a function~$g\colon\mathbf{N}\to\mathbf{N}$ such that every $k$-near-bipartite
graph~$G$ contains a set~$X$ of at most~$g(k)$ vertices such that~$G-X$ is bipartite.
This conjecture was confirmed by Reed~\cite{Ree99} in~1999.

Here is yet another way to weaken \Cref{que-half}: it was raised by Erd\H os
and Hajnal~\cite{ErHa66} and gave rise to a ``deep theorem'' (to quote
Gyárfás~\cite{Gya97}) demonstrated by Folkman~\cite{Fol69} in~1969.

\begin{restatable}[Folkman~\cite{Fol69}]{theorem}{folkman}\label{th:folkman}
Let~$k$ be a non-negative integer and~$G$ a graph.
    If for every induced subgraph~$H$ of~$G$, there exists a set~$Y\subseteq
V(H)$ of at most~$k$ vertices such that~$H-Y$ is half-stable,
then~$\chi(G)\leq k+2$.
\end{restatable}

The arguments developed by Folkman to demonstrate \Cref{th:folkman} are nice
and interesting.  However, they are difficult to access and it seems they
deserve more visibility.  The essence of Folkman's argument is not obvious to
see amongst the seventeen lemmas composing its proof, which is witnessed by the
fact that a simpler proof was asked for by Claude Tardif at a workshop  at the
University of Illinois in~2008~\cite{REGS08}.  Our goal is to present a much
more compact and shorter proof of \Cref{th:folkman}, which still relies on
Folkman's original idea, but with refined, factorised and sometimes generalised
statements.

We define the \emph{potential}~$\rho(H)$ of a graph~$H$ to be
$|V(H)|-2\alpha(H)+2$. Note that if a graph has potential~$k+2$ then one
can remove~$k$ vertices (but not fewer) to obtain a half-stable graph.  Given
a graph~$G$, let~$f(G)$ be the maximum of~$\rho(H)$ over all induced
subgraphs~$H$ of~$G$.  Note that if an induced subgraph~$H$ satisfies the property
mentioned in Theorem~\ref{th:folkman}, then~$\rho(H)\le k+2$, and hence
Theorem~\ref{th:folkman} can be reformulated
as~$\chi(G)\leq f(G)$ for every graph~$G$.  An induced
subgraph~$H$ of~$G$ is a \emph{witness} if $\rho(H)=f(G)$.
A $k$-colouring of a graph~$G$ is a mapping~$\varphi\colon V(G)\to\{1,2,\dotsc,k\}$
such that no two adjacent vertices have the same image under~$\varphi$.
If~$X$ is a subset of vertices of a graph~$G$, then~$G[X]$ is the subgraph of~$G$
induced by the vertices in~$X$ and~$G-X$ is the subgraph of~$G$ induced by
the vertices not in~$X$.

To demonstrate \Cref{th:folkman}, we use a simple, seemingly unrelated result,
which is a special case of a theorem proved by Hajnal~\cite{Haj65} in~1965.
A proof is included for completeness (the formulation and the argument are different
from what is customary seen). 

\begin{proposition}[Hajnal~\cite{Haj65}]\label{lem:hellylikestables}
    If a graph~$G$ admits an independent set containing more than half of the vertices,
    then there exists a vertex contained in all maximum independent sets of~$G$.
\end{proposition}

\begin{proof}
We proceed by induction on the number~$e$ of edges of~$G$.
The statement is true if~$e$ is zero. Assume that~$e\ge1$ and the statement is true for graphs with fewer than~$e$
edges. Note that the statement is true if~$G$ has an isolated vertex, and hence
    we may assume by induction that~$G$ is connected.  If there is an edge~$e$
    such that~$\alpha(G-e)=\alpha(G)$, then by induction a vertex~$v$ is
    contained in all maximum independent sets of~$G-e$, and hence also in all
    those of~$G$ since each of them is also a maximum independent set of~$G-e$. 
Thus it is enough to show that such an edge $e$ exists. 
Arguing by contradiction, let us assume that~$\alpha(G-e)>\alpha(G)$ for every edge~$e$ of~$G$. 

Let~$x$ be a vertex of~$G$ contained in the largest number of maximum
independent sets of~$G$.  Let~$G'$ be the graph obtained from~$G$ by
deleting~$x$ and all its neighbours. We show that the induction hypothesis can
be applied to~$G'$. Notice that adding~$x$ to any independent set of~$G'$
yields an independent set of~$G$, and hence~$\alpha(G')\le\alpha(G)-1$.  It
follows that~$\alpha(G')=\alpha(G)-1$ because~$x$ is contained in at least one
maximum independent set of~$G$.  Since~$G$ is connected and has more than one
vertex, we know that~$|V(G')|\le|V(G)|-2$, and therefore the induction
hypothesis applies to~$G'$, ensuring the existence of a vertex~$y$ of~$G'$
contained in every maximum independent set of~$G'$.  Consequently, $y$ is
contained in every maximum independent set of~$G$ that contains~$x$. The
definition of~$x$ thus ensures that a maximum independent of~$G$ contains~$x$
if and only if it contains~$y$.

To conclude, let~$z$ be a neighbour of~$x$. Since~$\alpha(G-xz)>\alpha(G)$,
there exists a maximum independent set~$I_x$ of~$G$ that contains~$x$ and such
that the only neighbour of~$z$ in~$I_x$ is~$x$. However, this implies
    that~$(I_x\setminus\{x\})\cup\{z\}$ is a maximum independent set of~$G$
containing~$y$ and not~$x$, a contradiction.
This concludes the proof.
\end{proof}

\section{Proof of \texorpdfstring{\Cref{th:folkman}}{Theorem 1.3}}\label{sec:proof} 

We proceed by contradiction. Let~$G$ be a counter-example to
Theorem~\ref{th:folkman} with the fewest vertices. In particular, $\chi(G) \geq
f(G)+1$. If $f(G)=2$, \Cref{obs-bip} implies that~$G$ is bipartite, a contradiction. It follows that $f(G) \geq 3$, and hence~$\chi(G)\ge4$. 

Note that every proper induced subgraph of~$G$ admits a $(\chi(G)-1)$-colouring. We
argue below that the graph satisfies a deeper critical property.

\begin{as}\label{lem:almost-critical}
For every subset~$X$ of vertices of~$G$,
\[\chi(G) \geq \chi(G[X])+\chi(G-X)-1.\]
\end{as}
\noindent
In particular,~\eqref{lem:almost-critical} implies that removing a
clique of size~$p$ in~$G$ results in a graph with chromatic number
either~$\chi(G)-p$ or~$\chi(G)-p+1$.

\begin{proof}[Proof of~\eqref{lem:almost-critical}]
The statement holds trivially if~$X$ or~$G-X$ is empty. Therefore,
    both~$G_1\coloneqq G[X]$ and~$G_2\coloneqq G-X$ have fewer
    vertices than~$G$. The minimality of~$G$ thus implies that none of~$G_1$
    and~$G_2$ is a counter-example to the statement of \Cref{th:folkman}.  For
    each~$i\in\{1,2\}$, let~$H_i$ be a witness of~$G_i$.  We
    have~$\rho(H_i)=f(G_i) \geq \chi(G_i)$.  We observe that $f(G) \geq
    \rho(G[V(H_1) \cup V(H_2)]) \geq \rho(H_1)+\rho(H_2)-2$, and hence~$f(G)
    \geq \chi(G_1)+\chi(G_2)-2$.  Since~$\chi(G) \geq f(G)+1$ by assumption,
the conclusion follows.
\end{proof}

\begin{as}\label{lem:evenholefree}
The graph~$G$ has no induced even cycle.
\end{as}
\begin{proof}
We proceed by contradiction. Assume that~$C$ is an induced even cycle of~$G$,
    and let~$2p$ be its length. Let~$A$ and~$B$ be the two maximum independent
    sets of~$C$.  We consider the graph~$G'$ obtained from~$G$ by merging all
    vertices of~$A$ in a single vertex~$a$ and all vertices of~$B$ in a single
    vertex~$b$, replacing every multi-edge that arises by a single edge (see
    Figure~\ref{fig:IllustrationA}).
     
Note that~$G'$ has fewer vertices than~$G$ and thus satisfies
    the conclusion of \Cref{th:folkman}.  It follows that 
    \[f(G') \geq \chi(G') \geq \chi(G) \geq f(G)+1.\]
    
    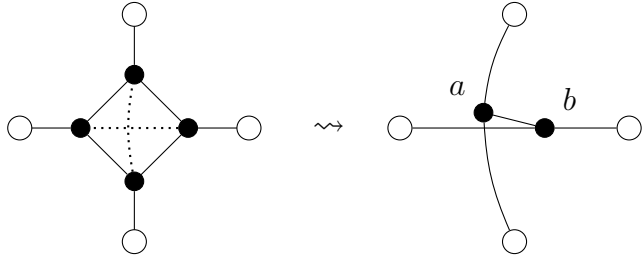
\begin{figure}[h]
  \centering
  \begin{tikzpicture}
    \tikzstyle{whitenode}=[draw,circle,fill=white,minimum size=9pt,inner sep=0pt]
    \tikzstyle{ghost}=[draw=white,circle,inner sep=0pt]
    \tikzstyle{blacknode}=[draw,circle,fill=black,minimum size=7pt,inner sep=0pt]
    \tikzstyle{rednode}=[draw,circle,fill=red,minimum size=9pt,inner sep=0pt]
    \tikzstyle{greennode}=[draw,circle,fill=green,minimum size=9pt,inner sep=0pt]
    \tikzstyle{bluenode}=[draw,circle,fill=black!20!blue,minimum size=9pt,inner sep=0pt]
    \tikzstyle{texte}=[minimum size=0pt,inner sep=0pt]
\draw (0,0) node[blacknode] (a1) {}
--++ (45:1cm) node[blacknode] (b1) {}
--++ (180-45:1cm) node[blacknode] (a2) {}
--++ (180+45:1cm) node[blacknode] (b2) {};
\draw (a1) -- (b2);
\draw (a1)
--++ (-90:0.8cm) node[whitenode] (a1n) {};
\draw (b1)
--++ (0:0.8cm) node[whitenode] (b1n) {};
\draw (a2)
--++ (90:0.8cm) node[whitenode] (a2n) {};
\draw (b2)
--++ (180:0.8cm) node[whitenode] (b2n) {};
\draw[bend left=10,dotted,thick] (a1) edge node {} (a2);
\draw[dotted,thick] (b1) -- (b2);

\draw (b1)
++ (0:1.85cm) node (x) {$\leadsto$};

\draw (a1)
++ (0:5cm) node[ghost] (ta1) {}
++ (45:1cm) node[ghost] (tb1) {}
++ (180-45:1cm) node[ghost] (ta2) {}
++ (180+45:1cm) node[ghost] (tb2) {};
\draw (ta1)
++ (-90:0.8cm) node[whitenode] (ta1n) {};
\draw (tb1)
++ (0:0.8cm) node[whitenode] (tb1n) {};
\draw (ta2)
++ (90:0.8cm) node[whitenode] (ta2n) {};
\draw (tb2)
++ (180:0.8cm) node[whitenode] (tb2n) {};

\draw (b1)
++ (3:3.89cm) node[blacknode][label=90+45:$a$] (a) {};
\draw (b1)
++ (0:4.69cm) node[blacknode][label=45:$b$] (b) {};

\draw[bend left=10] (a) edge node {} (ta2n);
\draw[bend right=10] (a) edge node {} (ta1n);
\draw (a) -- (b) -- (tb1n);
\draw (b) edge node {} (tb2n);
  \end{tikzpicture}
  \caption{An example of the reduction from~$G$ to~$G'$ for~(\ref{lem:evenholefree}).}\label{fig:IllustrationA}
\end{figure}

Let~$H'$ be a witness of~$G'$, so~$\rho(H')\geq f(G)+1$. 
    Let us consider the subgraph~$H$ of~$G$ induced by~$(V(H')\setminus\{a,b\})\cup V(C)$.
The potential~$\rho(H)$ of~$H$ can be bounded as follows. 
\begin{align}\label{eq:fG}
    f(G) &\geq \rho(H)\notag\\
    &\geq \rho(H') + |V(C)|-2 - 2 (\alpha(H)-\alpha(H'))\notag\\
    &\geq f(G) + 2p-1 - 2 (\alpha(H)-\alpha(H')).
\end{align}

Clearly, $\alpha(H)-\alpha(H') \leq \alpha(C)=p$. 
In fact, we must have $\alpha(H)-\alpha(H') \leq p-1$, as we now explain. 
    Let~$I$ be a maximum independent set of~$H$. If~$|I\cap V(C)|\le p-1$,
    then~$I\cap V(H')$ is an independent set of~$H'$ of size at least~$\alpha(H)-(p-1)$.
    Otherwise, $I\cap V(C)\in\{A,B\}$, and hence one of~$(I\setminus V(C))\cup\{a\}$
    and~$(I\setminus V(C))\cup\{b\}$ is an independent set of~$H$ of size at least~$\alpha(H)-(p-1)$.

Plugging the inequality~$\alpha(H)-\alpha(H') \leq p-1$ into~\eqref{eq:fG} yields that
\[
f(G) \geq f(G)+2p-1 - 2 (p-1) > f(G), 
\]
which is the desired contradiction. 
\end{proof}

\noindent
We call \emph{diamond} the graph obtained from the complete graph on~$4$ vertices
by deleting an edge (see the graph induced by $u,v,x,y$ in Figure~\ref{fig:IllustrationC}).
\begin{as}\label{lem:nodiamond}
No induced subgraph of~$G$ is a diamond.
\end{as}

\begin{proof}
Let~$xy$ be an edge of~$G$. Let~$A$ be the set of common neighbors of~$x$
    and~$y$, and assume for a contradiction that~$A$ does not induce a clique.
    For every two distinct vertices~$u$ and~$v$ in~$A$ such that $uv \not\in
    E(G)$, let~$G_{uv}$ be the graph obtained from~$G-\{x,y\}$ by merging the
    two vertices~$u$ and~$v$ into a vertex, once again replacing every multi-edge
    that arises by a single edge (see Figure~\ref{fig:IllustrationC}).
    
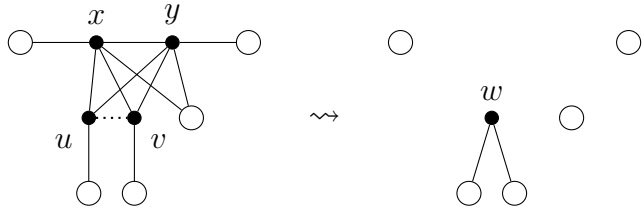
\begin{figure}[h] 
  \centering
  \begin{tikzpicture}
    \tikzstyle{whitenode}=[draw,circle,fill=white,minimum size=9pt,inner sep=0pt]
    \tikzstyle{ghost}=[draw=white,circle,inner sep=0pt]
    \tikzstyle{blacknode}=[draw,circle,fill=black,minimum size=5pt,inner sep=0pt]
    \tikzstyle{rednode}=[draw,circle,fill=red,minimum size=9pt,inner sep=0pt]
    \tikzstyle{greennode}=[draw,circle,fill=green,minimum size=9pt,inner sep=0pt]
    \tikzstyle{bluenode}=[draw,circle,fill=black!20!blue,minimum size=9pt,inner sep=0pt]
    \tikzstyle{texte}=[minimum size=0pt,inner sep=0pt]
\draw (0,0) node[blacknode][label=90:$x$] (x) {}
--++ (0:1cm) node[blacknode][label=90:$y$] (y) {};
\draw (x)
--++(-63.43:1.118cm) node [blacknode][label=-45:$v$] (v) {};
\draw (y) -- (v);

\draw (v)
++(0:0.75cm) node [whitenode] (w) {};
\draw (v)
++(180:0.6cm) node [blacknode][label=-90-45:$u$] (u) {};
\draw[dotted, thick] (u) -- (v);
\draw (x) -- (u) -- (y) -- (w) -- (x);
\draw (u)
--++(-90:1cm) node [whitenode] (un) {};
\draw (v)
--++(-90:1cm) node [whitenode] (vn) {};
\draw (x)
--++ (180:1cm) node [whitenode] (xn) {};
\draw (y)
--++ (0:1cm) node [whitenode] (yn) {};

\draw (v)
++ (0:2.5cm) node (z) {$\leadsto$};

\draw (x)
++ (0:5cm) node[ghost] (tx) {}
++ (0:1cm) node[ghost] (ty) {};
\draw (tx)
++(-63.43:1.118cm) node [ghost] (tv) {};

\draw (tv)
++(0:0.75cm) node [whitenode] (tw) {};
\draw (tv)
++(180:0.6cm) node [ghost] (tu) {};
\draw (tu)
++(-90:1cm) node [whitenode] (tun) {};
\draw (tv)
++(-90:1cm) node [whitenode] (tvn) {};
\draw (tx)
++ (180:1cm) node [whitenode] (txn) {};
\draw (ty)
++ (0:1cm) node [whitenode] (tyn) {};

\draw (u)
++ (0:5.3cm) node[blacknode][label=90:$w$] (tuv) {};
\draw (tun) -- (tuv) -- (tvn);

  \end{tikzpicture}
  \caption{An example of the reduction from $G$ to $G_{uv}$ for (\ref{lem:nodiamond}).}\label{fig:IllustrationC}
\end{figure}
\bigskip

\noindent
We prove the following statement.
\begin{sas}\label{cl:chi2}
For every two distinct vertices~$u$ and~$v$ in~$A$ that are not adjacent in~$G$,
\begin{enumerate}[(1)]
    \item\label{prop1}
    $\chi(G_{uv}) = \chi(G)-2$; and  
    \item\label{prop2}
    in every $(\chi(G)-2)$-colouring of~$G_{uv}$, all colours appear on~$A$.
\end{enumerate}
\end{sas}

\begin{proof}
Let~$w$ be the vertex resulting from the identification of~$u$ and~$v$.  We
    first prove~\eqref{prop1}.  Let~$H$ be a witness of~$G_{uv}$. If~$H$
    contains~$w$, then we let~$H'$ be the subgraph of~$G$ induced by~$(V(H) \setminus
    \{w\}) \cup \{u,v,x,y\}$.  If~$H$ does not contain~$w$, then we let~$H'$ be
    the subgraph of~$G$ induced by~$V(H) \cup \{u,x,y\}$.  In either case, we observe
    that~$\rho(H')>\rho(H)$ because~$\alpha(H')\le\alpha(H)+1$. Consequently,
    we derive that $f(G_{uv})\leq f(G)-1 \leq \chi(G)-2$.

Because~$G_{uv}$ has fewer vertices than~$G$, we have $f(G_{uv})\geq
    \chi(G_{uv})$, and hence $\chi(G_{uv}) \leq \chi(G)-2$.  On the other hand,
    $\chi(G_{uv}) \geq \chi(G)-2$ also holds (from a $\chi(G_{uv})$-colouring
    of~$G_{uv}$, it suffices to keep the same colour on the vertices belonging
    to both graphs, assign the colour of~$w$ to each of~$u$ and~$v$, and the
    colours~$\chi(G_{uv})+1$ and~$\chi(G_{uv})+2$ to~$x$ and~$y$, respectively,
    to obtain a~$(\chi(G_{uv})+2)$-colouring of~$G$).  Therefore, $\chi(G_{uv})
    = \chi(G)-2$.

Let us now prove~\eqref{prop2}.  Suppose, on the contrary, that there is a 
$(\chi(G)-2)$-colouring~$\varphi$ of~$G_{uv}$ contradicting~\eqref{prop2},
    and let~$c\in\varphi(V(G_{uv}))\setminus\varphi(A)$.  We derive from~$\varphi$ a
    $(\chi(G)-1)$-colouring of~$G$, which is a contradiction.
    Set~$c'=\chi(G)-1$, so~$c'\notin\varphi(V(G_{uv}))$.  We
    define~$\phi$ to be the colouring of~$G$ obtained from~$\varphi$ by
    colouring~$u$ and~$v$ with~$\varphi(w)$, colouring~$x$ with~$c$ and~$y$
    with~$c'$, and changing the colour of every neighbour of~$x$ that belongs
    to~${\varphi}^{-1}(\{c\})$ to~$c'$.  As~${\varphi}^{-1}(\{c\})$
    is an independent set disjoint from~$A$, and all common neighbours of~$x$ and~$y$
    belong to~$A$, we infer that~$\phi$ is a proper $(\chi(G)-1)$-colouring
    of~$G$.\cqed
\end{proof}
 
We now consider the graph~$G'$ obtained from~$G-\{x,y\}$ by creating a
    vertex~$z$ with neighbourhood~$A$.  We observe that the existence of a
    $(\chi(G)-2)$-colouring of~$G'$ would contradict~\eqref{cl:chi2}.
    Indeed,~\eqref{cl:chi2} implies that~$|A|\ge\chi(G)-1$. It follows that
    if~$\varphi$ is a $(\chi(G)-2)$-colouring of~$G'$, then there exist two
    vertices~$u$ and~$v$ in~$A$ such that~$\varphi(u)=\varphi(v)$.
    Consequently,~$u$ and~$v$ are not adjacent in~$G$ and thus~$\varphi$ readily
    yields a $(\chi(G)-2)$-colouring of~$G_{uv}$ such that one colour,
    namely~$\varphi(z)$, does not appear on~$A$.  Therefore,~$\chi(G')\geq
    \chi(G)-1$, and hence~$f(G')\geq \chi(G') \geq f(G)$.

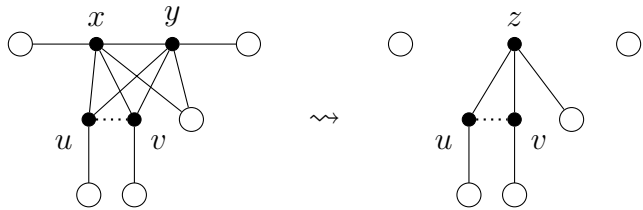
\begin{figure}[h] 
  \centering
  \begin{tikzpicture}
    \tikzstyle{whitenode}=[draw,circle,fill=white,minimum size=9pt,inner sep=0pt]
    \tikzstyle{ghost}=[draw=white,circle,inner sep=0pt]
    \tikzstyle{blacknode}=[draw,circle,fill=black,minimum size=5pt,inner sep=0pt]
    \tikzstyle{rednode}=[draw,circle,fill=red,minimum size=9pt,inner sep=0pt]
    \tikzstyle{greennode}=[draw,circle,fill=green,minimum size=9pt,inner sep=0pt]
    \tikzstyle{bluenode}=[draw,circle,fill=black!20!blue,minimum size=9pt,inner sep=0pt]
    \tikzstyle{texte}=[minimum size=0pt,inner sep=0pt]
\draw (0,0) node[blacknode][label=90:$x$] (x) {}
--++ (0:1cm) node[blacknode][label=90:$y$] (y) {};
\draw (x)
--++(-63.43:1.118cm) node [blacknode][label=-45:$v$] (v) {};
\draw (y) -- (v);

\draw (v)
++(0:0.75cm) node [whitenode] (w) {};
\draw (v)
++(180:0.6cm) node [blacknode][label=-90-45:$u$] (u) {};
\draw[dotted, thick] (u) -- (v);
\draw (x) -- (u) -- (y) -- (w) -- (x);
\draw (u)
--++(-90:1cm) node [whitenode] (un) {};
\draw (v)
--++(-90:1cm) node [whitenode] (vn) {};
\draw (x)
--++ (180:1cm) node [whitenode] (xn) {};
\draw (y)
--++ (0:1cm) node [whitenode] (yn) {};

\draw (v)
++ (0:2.5cm) node (z) {$\leadsto$};

\draw (x)
++ (0:5cm) node[ghost] (tx) {}
++ (0:1cm) node[ghost] (ty) {};
\draw (tx)
++(-63.43:1.118cm) node [blacknode][label=-45:$v$] (tv) {};

\draw (tv)
++(0:0.75cm) node [whitenode] (tw) {};
\draw (tv)
++(180:0.6cm) node [blacknode][label=-90-45:$u$] (tu) {};
\draw (tu)
--++(-90:1cm) node [whitenode] (tun) {};
\draw (tv)
--++(-90:1cm) node [whitenode] (tvn) {};
\draw (tx)
++ (180:1cm) node [whitenode] (txn) {};
\draw (ty)
++ (0:1cm) node [whitenode] (tyn) {};

\draw (x)
++ (0:5.5cm) node[blacknode][label=90:$z$] (txy) {};
\draw (tu) -- (txy) -- (tv);
\draw (txy) -- (tw);
\draw[dotted, thick] (tu) -- (tv);
  \end{tikzpicture}
  \caption{An example of the reduction from~$G$ to~$G'$ for~(\ref{lem:nodiamond}).}\label{fig:IllustrationC2}
\end{figure}    

Let~$H'$ be a witness of~$G'$, so~$\rho(H')\geq f(G)$. If~$z \in V(H')$, then
    the potential of the subgraph~$H_0$ of~$G$ induced by~$(V(H') \setminus \{z\}) \cup
    \{x,y\}$ is larger than that of~$H'$ because~$\alpha(H_0)\le\alpha(H')$. This
    is a contradiction since~$f(G)\le f(G')=\rho(H')$. In particular,~$H'$ is
    thus an induced subgraph of~$G-\{x,y\}$ and consequently we deduce
    that~$f(G)=\rho(H')$.  We also deduce that the potential of the subgraph
    of~$G'$ induced by~$V(H') \cup \{z\}$ is less than that of~$H'$, so there
    exists a maximum independent set of~$H'$ that is disjoint from~$A$.  If
    there is a vertex~$w \in A\setminus V(H')$, then the potential of the
    subgraph of~$G$ induced by~$V(H') \cup \{w,x,y\}$ is larger than that
    of~$H'$, a contradiction.
    Therefore~$A \subseteq V(H')$, and~$z \not\in V(H')$.

Let~$U = V(H') \setminus A$. We note that $|U|=|V(H')|-|A|$.
    From~\eqref{cl:chi2}, we derive that $|A|\geq \chi(G)-1 \geq f(G)=\rho(H')$.
    Therefore, $|U|\leq 2 \alpha(H')-2$. Set~$G_U = H'-A$, so~$G_U$ is
    the subgraph of~$G$ induced by~$U$.
It follows from the previous paragraph
    that~$\alpha(G_U)=\alpha(H')$, which is greater than~$\frac12\cdot|U|$.
    Proposition~\ref{lem:hellylikestables} therefore implies the existence of a
    vertex~$w$ in~$U$ such that every independent set of~$G_U$ that has
    size~$\alpha(H')$ contains~$w$.  

We consider the subgraph~$H$ of~$G$ induced by~$(V(H') \setminus \{w\}) \cup
    \{x,y\}$.  Since its potential is at most~$f(G)$, there is a maximum
    independent set~$I$ in~$H$ of size~$\alpha(H')+1$. We observe that~$I$
    contains exactly one of~$\{x,y\}$, say~$x$ without loss of generality. We
    derive that~$I \setminus \{x\}$ is disjoint from~$A$, which is contained in
    the neighbourhood of~$x$. We note that $I \setminus \{x\}$ is an
    independent set of size $\alpha(H')=\alpha(G_U)$ that is contained in~$U$
    and does not contain~$w$, a contradiction.
\end{proof}

Equipped with all these properties, we may now finish the proof of
Theorem~\ref{th:folkman}. We start by proving that~$G$ has no triangle.
Choose a maximum clique~$K$ of~$G$, and let~$\omega$
be its size. We let~$u_1,u_2, \dotsc, u_\omega$ be the vertices of~$K$ and we
set~$G_K = G-K$.  For every~$i\in \{1,2,\dotsc,\omega\}$, we let~$N_i$ be
the neighbourhood of~$u_i$ in~$G_K$. By~\eqref{lem:nodiamond}, and because~$K$
is maximum, the sets~${(N_i)}_{1\le i\le \omega}$ are pairwise disjoint.
Moreover,~\eqref{lem:evenholefree} implies that they are pairwise non-adjacent,
that is, if~$G$ has an edge with one end-vertex in~$N_i$ and the other
in~$N_j$, then~$i=j$.

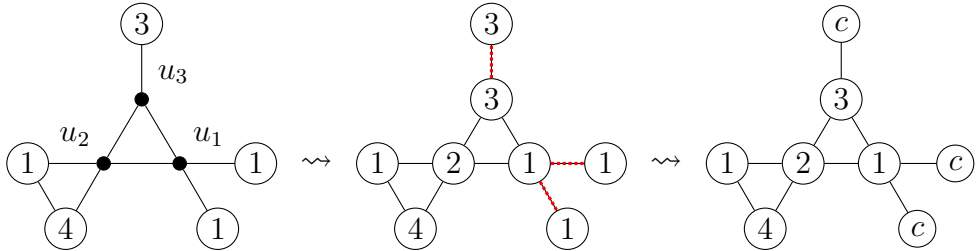
\begin{figure}[h] 
  \centering
  \begin{tikzpicture}
    \tikzstyle{whitenode}=[draw,circle,fill=white,minimum size=9pt,inner sep=0pt]
    \tikzstyle{ghost}=[draw=white,circle,inner sep=0pt]
    \tikzstyle{blacknode}=[draw,circle,fill=black,minimum size=5pt,inner sep=0pt]
    \tikzstyle{rednode}=[draw,circle,fill=red,minimum size=9pt,inner sep=0pt]
    \tikzstyle{greennode}=[draw,circle,fill=green,minimum size=9pt,inner sep=0pt]
    \tikzstyle{bluenode}=[draw,circle,fill=black!20!blue,minimum size=9pt,inner sep=0pt]
    \tikzstyle{texte}=[minimum size=0pt,inner sep=0pt]
\draw (0,0) node[blacknode][label=110:$u_2$] (v2) {}
--++ (0:1cm) node[blacknode][label=70:$u_1$] (v1) {};
\draw (0.5,0.85) node [blacknode][label=45:$u_3$] (v3) {};
\draw (v1) -- (v3) -- (v2);
\draw (v3)
--++ (90:1cm) node[whitenode] (nv3) {$\ 3\ $};
\draw (v1)
--++ (0:1cm) node[whitenode] (n1v1) {$\ 1\ $};
\draw (v1)
--++ (-60:1cm) node[whitenode] (n2v1) {$\ 1\ $};
\draw (v2)
--++ (180:1cm) node[whitenode] (n1v2) {$\ 1\ $};
\draw (v2)
--++ (180+60:1cm) node[whitenode] (n2v2) {$\ 4\ $};
\draw (n1v2) -- (n2v2);

\draw (v1)
++ (0:1.8cm) node (z) {$\leadsto$};

\draw (4.6,0) node[whitenode] (dv2) {$\ 2\ $}
--++ (0:1cm) node[whitenode] (dv1) {$\ 1\ $};
\draw (0.5+4.6,0.85) node [whitenode] (dv3) {$\ 3\ $};
\draw (dv1) -- (dv3) -- (dv2);
\draw (dv3)
--++ (90:1cm) node[whitenode] (dnv3) {$\ 3\ $};
\draw (dv1)
--++ (0:1cm) node[whitenode] (dn1v1) {$\ 1\ $};
\draw (dv1)
--++ (-60:1cm) node[whitenode] (dn2v1) {$\ 1\ $};
\draw (dv2)
--++ (180:1cm) node[whitenode] (dn1v2) {$\ 1\ $};
\draw (dv2)
--++ (180+60:1cm) node[whitenode] (dn2v2) {$\ 4\ $};
\draw[densely dotted, very thick,red] (dv3) -- (dnv3);
\draw[densely dotted, very thick,red] (dv1) -- (dn1v1);
\draw[densely dotted, very thick,red] (dv1) -- (dn2v1);

\draw (dv1)
++ (0:1.8cm) node (dz) {$\leadsto$};

\draw (9.2,0) node[whitenode] (tv2) {$\ 2\ $}
--++ (0:1cm) node[whitenode] (tv1) {$\ 1\ $};
\draw (0.5+9.2,0.85) node [whitenode] (tv3) {$\ 3\ $};
\draw (tv1) -- (tv3) -- (tv2);
\draw (tv3)
--++ (90:1cm) node[whitenode] (tnv3) {$\ c\ $};
\draw (tv1)
--++ (0:1cm) node[whitenode] (tn1v1) {$\ c\ $};
\draw (tv1)
--++ (-60:1cm) node[whitenode] (tn2v1) {$\ c\ $};
\draw (tv2)
--++ (180:1cm) node[whitenode] (tn1v2) {$\ 1\ $};
\draw (tv2)
--++ (180+60:1cm) node[whitenode] (tn2v2) {$\ 4\ $};

\draw (dn1v2) -- (dn2v2);
\draw (tn1v2) -- (tn2v2);
  \end{tikzpicture}
  \caption{An example of the colouring extension when $\omega(G) \geq 3$.}\label{fig:IllustrationClique}
\end{figure}  

Suppose, for a contradiction, that~$\omega\ge3$.
We deduce from~\eqref{lem:almost-critical} that $\chi(G_K) \leq \chi(G)-2$.
Let~$\varphi$ be a $(\chi(G)-2)$-colouring of~$G_K$.  We argue how to obtain
from~$\varphi$ a $(\chi(G)-1)$-colouring of~$G$, which would be a contradiction.
Set~$c = \chi(G)-1$, so~$c\notin\varphi(V(G_K))$. Colour every
vertex $v$ not in~$K$ with~$\varphi(v)$.  For each~$i\in\{1,2,\dotsc,\omega\}$, colour~$u_i$
with colour~$i$, and next recolour every vertex in~$N_i\cap{\varphi}^{-1}(\{i\})$
with colour~$c$.  Since the sets~${(N_i)}_{1\le i\le \omega}$ are pairwise
non-adjacent and we recolour an independent set in each set~$N_i$, we obtain a
proper colouring of~$G$. And since~$\chi(G)-1\ge\omega$, because~$\chi(G)>f(G)\ge\omega$,
this colouring uses less than~$\chi(G)$ colours, a contradiction.
Therefore, the graph~$G$ contains no triangle.

Consider a shortest cycle~$C$ in~$G$ (there is one since $\chi(G)> f(G) \geq
2$). By~\eqref{lem:evenholefree} and the previous argument, we have
$|V(C)|=2p+1$ where $p \geq 2$. Let $v_1,v_2,\dotsc, v_{2p+1}$ be the vertices
on~$C$, consecutively. Let~$N'_i$ be the set of neighbours of~$v_i$ not in~$C$.
Note that every vertex not in~$C$ has at most one neighbour in~$C$, for
otherwise~$G$ would contain either an odd cycle shorter than~$C$ or an induced
even cycle, thereby contradicting~\eqref{lem:evenholefree}.  Consequently, the
sets~${(N'_i)}_{1\le i\le 2p+1}$ are pairwise disjoint.  It also follows
from~\eqref{lem:evenholefree} that the sets~${(N'_i)}_{1\le i\le 2p+1}$ are
pairwise non-adjacent. 

Let~$G_C = G-C$.
By~\eqref{lem:almost-critical}, we obtain a $(\chi(G)-2)$-colouring~$\varphi$
of~$G_C$. Similarly as before, we argue how to deduce from~$\varphi$ a
$(\chi(G)-1)$-colouring~$\phi$ of~$G$, hence obtaining the final
contradiction. Recall that~$\chi(G)\ge4$.  For~$i \in \{1,2,\dotsc,p\}$, we
colour~$v_{2i}$ with colour~$1$ and~$v_{2i-1}$ with colour~$2$. We
colour~$v_{2p+1}$ with colour~$3$.  Set~$c = \chi(G)-1$,
so~$c\notin\varphi(V(G_C))$.  For each~$i\in\{1,2,\dotsc,2p+1\}$, every vertex
in~$N'_i\cap {\varphi}^{-1}(\phi(v_i))$ is coloured with colour~$c$. We
define~$\phi$ to be equal to~$\varphi$ on all the remaining vertices of~$G$.
Similarly as before, the properties of the sets~${(N'_i)}_{1\le i\le 2p+1}$
ensure that~$\phi$ is a proper~$(\chi(G)-1)$-colouring of~$G$.  This
contradiction concludes the proof of the theorem.  

\section{Conclusion} 

With now a clear understanding of why Theorem~\ref{th:folkman} holds, it is tempting to look for a strengthening of the theorem. We already discussed in Section~\ref{sec:intro} some natural generalisations of the statement that unfortunately do not hold. 
In this section, we conclude this work with a discussion of other potential generalisations. 
Let us first restate Theorem~\ref{th:folkman} as follows, where the notation~$H \subseteq_i G$ means that~$H$ is an induced subgraph of~$G$. 

\begin{usethmcounterof}{th:folkman}[Folkman~\cite{Fol69}]
For every graph~$G$,
    \begin{equation}\label{eq-folkman}
\chi(G) \leq \max_{H \subseteq_i G} ( |V(H)|-2\cdot(\alpha(H)-1)).
    \end{equation}
\end{usethmcounterof}

The parameter~$\alpha(H)$ can be interpreted as ``the size of the largest induced subgraph of~$H$ with chromatic number~$1$''. 
This suggests a new approach for a generalisation. For every positive integer~$p$, define~$\alpha_p(H)$ as the size of a largest induced subgraph of~$H$ with chromatic number at most~$p$.

\begin{question}\label{que-HigherChi}
Given a positive integer~$p$, does there exist~$c_p > 1$ such that for every graph~$G$,
\[\chi(G) \leq \max_{H \subseteq_i G} \big( |V(H)|-c_p\cdot(\alpha_p(H)-p)\big)?\]
\end{question}

Note that the inequality holds for~$c_p \leq 1$ by taking~$H=G$. Theorem~\ref{th:folkman} answers the case~$p=1$ of Question~\ref{que-HigherChi} in the affirmative, showing that~$c_1$ can even be taken as high as~$2$. 
Before considering larger values of~$p$, let us point out that~$c_1=2$ is best possible. 
Indeed, suppose that~$c_1>2$ and take~$G$ to be the odd cycle~$C_5$.
Because~$\chi(C_5)=3$ and~$c_1>2$, the maximum in the right side of~\eqref{eq-folkman} must
be attained by~$H$ being the null graph (that is, $V(H)=\varnothing$), implying that $c_1
\geq 3$.  However, the Mycielski graph~$M_{c'_1}$, where~$c'_1 =
\lfloor c_1 \rfloor$, now contradicts the desired inequality.
Since~$\chi(M_{c'_1})=c'_1+1$, taking~$H$ to be the null graph is not
sufficient to ensure the inequality. Moreover, as~$c'_1\ge3$ we know
that~$\mir(M_{c'_1})\ge\frac1{c'_1}$ (as mentioned in the introduction), and
hence~$\alpha(H)\ge|V(H)|/c'_1$ whenever~$H$ is a non-null induced subgraph
of~$M_{c'_1}$. It follows that, in this case, $|V(H)|-c_1(\alpha(H)-1) \leq
|V(H)|-c'_1(\alpha(H)-1)\le c'_1<\chi(M_{c'_1})$, a contradiction.  This shows
that~$c_1=2$ is best possible. 

For larger values of~$p$, it turns out that the answer to
Question~\ref{que-HigherChi} is always negative. First note that since
$\alpha_{p+1}(H) \geq \alpha_p(H)+1$ in any graph~$H$ with chromatic number
higher than~$p$, a negative answer for a positive integer~$p_0$ implies a negative answer for
every integer~$p\ge p_0$. We now explain why the answer is negative for~$p=2$.

For simplicity, we define~$f_2(G)$ to be 
\[
\max_{H \subseteq_i G} \big( |V(H)|-c_2\cdot(\alpha_2(H)-2)\big). 
\]
Let~$\ell \geq 2$.  For~$k\geq 2$, let~$M'_{k,\ell}$ be obtained by
applying~$k-2$ times the Mycielski construction to the cycle on~$2\ell+1$
vertices. Note that~$M'_{k,2}$ is the standard Mycielski graph~$M_k$. 

Suppose by contradiction that~$c_2>1$. 
We prove by induction on~$k$ that for every~$k \geq 1$, we have~$c_2 \geq \frac{k+1}2$. 
For~$k=1$, this is true by our assumption that~$c_2>1$. 
For~$k\geq 2$, we consider~$f_2(M'_{k,\ell})$ for~$\ell$ large enough in terms of~$k$. 

If the maximum in the definition of~$f_2(M'_{k,\ell})$ is attained on the null
graph, then we obtain~$2 c_2\geq \chi(M'_{k,\ell})=k+1$ as desired.  Thus we
may assume this is not the case.  If the maximum is attained on a proper
non-null induced subgraph~$H$, then~$H$ is $k$-colourable and thus
satisfies~$\alpha_2 (H) \geq \frac{2 |V(H)|}{k}$.  Since~$H$ gives a higher
value than that given by the null graph in the definition
of~$f_2(M'_{k,\ell})$, we deduce that 
\[
|V(H)|-c_2\cdot(\alpha_2(H)-2) > 2c_2. 
\]
However, using that~$c_2 \geq \frac{k}{2}$ (by induction), we obtain~$\alpha_2
(H) < \frac{2 |V(H)|}{k}$, a contradiction. 

It follows that the maximum in the definition of~$f_2(M'_{k,\ell})$ is attained by~$M'_{k,\ell}$ itself. 
Observe that $|M'_{k,\ell}|=(\ell+1)\cdot 2^{k-1}-1$ and $\alpha_2 (M'_{k,\ell})=\ell \cdot 2^{k-1}$. 
Consequently, \[
f_2(M'_{k,\ell})=
(\ell+1)\cdot 2^{k-1}-1
- c_2 \cdot (\ell \cdot 2^{k-1}-2)
\geq \chi(M'_{k,\ell}) = k+1,  
\]
which does not hold for~$\ell$ large enough in terms of~$k$, since $c_2>1$. 
Therefore, this last case cannot occur either, and~$c_2 \geq \frac{k+1}2$ holds, as announced.

In summary, we have shown that~$c_2>1$ implies that~$c_2 \geq \frac{k+1}2$ for
every positive integer~$k$, and hence that~$c_2$ is unbounded, a contradiction.  We
conclude that~$c_2 \leq 1$ must hold.

Given this state of affairs, it was suggested to one of the
authors to try and replace~$\alpha_2(H)$ with the maximum size of an induced
subgraph of~$H$ isomorphic to a bipartite cograph.  However, the answer remains
negative even when considering the maximum size of an induced subgraph of~$H$
isomorphic either to a complete bipartite graph with one side of size at
most~$2$ or to~$P_1+P_2$, the graph on three vertices with only one edge. To
see this, it suffices to apply Mycielski's construction once on a large clique,
and to do the arithmetic for various cases depending on the structure of a
subgraph attaining maximum potential (whether it contains more vertices of the
clique than of the maximum stable set, and whether it contains the unique
vertex whose neighbourhood is a stable set). While this supersedes the above
argument that~$c_2 \leq 1$, the proof is decidedly non-illuminating, and we
refrain from including it here. Considering how restricted the replacement
for~$\alpha_2(H)$ is here, generalising Theorem~\ref{th:folkman} in this
direction does not seem feasible either.

Our original motivations for finding a short proof of Theorem~\ref{th:folkman}
were to make sense of the statement and understand what bigger truth it could be
part of. With the new-found insight, Theorem~\ref{th:folkman} seems all the
more to be a truly isolated statement, almost a singularity in the realm of
graph colouring.

\section*{Acknowledgements} 
This research was started at the Graph Theory meeting in Oberwolfach in
January~2019. Thanks to the organizers and to the other workshop participants
for creating a productive working atmosphere. 
Thanks to Paul Seymour for stimulating discussions on this topic. 


\end{document}